\documentclass[12pt,a4paper]{amsart}

\usepackage{amssymb,amsmath,amsthm}
\usepackage{euscript}
\usepackage{mathrsfs}
\usepackage{hyperref}
\hypersetup{
pdftitle={Research Paper},
pdfauthor={Abbas {Nasrollah Nejad} -Ashkan Nikseresht-Ali Akbar {Yazdan Pour}-Rashid Zaare-Nahandi},
pdfsubject={Tame Graphs and Clutters},
pdfcreator={Abbas Nasrollah Nejad},
colorlinks,
linkcolor=blue,
citecolor=magenta,
anchorcolor=red,
bookmarksopen,
urlcolor=red,
filecolor=red,
pdfpagetransition={Wipe}
}

\theoremstyle{plain}
\newtheorem{thm}{Theorem}[section]
\newtheorem{cor}[thm]{Corollary}
\newtheorem{lem}[thm]{Lemma}
\newtheorem{prop}[thm]{Proposition}

\newtheorem{Proposition}[thm]{Proposition}

\theoremstyle{definition}

\newtheorem{defn}[thm]{Definition}
\newtheorem{ex}[thm]{Example}

\theoremstyle{remark}

\newtheorem{rem}{Remark}

\newtheorem{nota}{Notation}

\newcommand{\rar}{\rightarrow}

\newcommand{\surjects}{\twoheadrightarrow}
\newcommand{\injects}{\hookrightarrow}

\def\spec{{\rm spec}\,}

\def\ker{{\rm ker}\,}

\def\spec#1{{\rm Spec}\left( #1 \right)}
\def\proj#1{{\rm Proj}\, (#1)}
\def\supp#1{{\rm Supp}\, (#1)}

\newcommand{\cc}{\EuScript{C}}

\newcommand{\N}{\mathrm{N}}

\newcommand{\Bl}[1]{\mathrm{Bl}_{#1}(R)}

\newcommand{\n}{\mathbb{N}}
\newcommand{\z}{\mathbb{Z}}
\newcommand{\bbA}{\mathbb{A}}

\newcommand{\ifof}{if and only if }
\newcommand{\se}{\subseteq}

\newcommand{\sm}{\setminus }

\newcommand{\give}{$\Rightarrow$}
\newcommand{\rgive}{$\Leftarrow$}
\renewcommand{\l}{\left}
\renewcommand{\r}{\right}
\newcommand{\rg}{\rangle}
\renewcommand{\lg}{\langle}

\newcommand{\f}[2]{\frac{#1}{#2}}
\newcommand{\tohi}{\emptyset}

\begin{document}

\title[Tame Graphs, Clutters and their Rees algebras]{Tame Graphs, Clutters and their Rees algebras }
\author[A. Nasrollah Nejad, A. Nikseresht, A. A. {Yazdan Pour}, R. Zaare-Nahandi]
{Abbas Nasrollah nejad, Ashkan  Nikseresht, {Ali Akbar} {Yazdan Pour}, Rashid Zaare-Nahandi }
\address{
department of mathematics\\ institute for advanced studies in basic sciences (IASBS)\\
p.o.box 45195-1159 \\ zanjan, iran} \email{abbasnn@iasbs.ac.ir} \email{ashkan\_nikseresht@yahoo.com} \email{yazdan@iasbs.ac.ir}
\email{rashidzn@iasbs.ac.ir}


\begin{abstract}
A tame ideal is an ideal $I$ such that the blowup of the affine space $\mathbb{A}_k^n$ along $I$ is regular. In
this paper, we give a combinatorial characterization of tame squarefree monomial ideals. More precisely, we show that a square free monomial ideal is tame if and only if the corresponding clutter is a union of some isolated vertices and a complete $d$-partite $d$-uniform clutter. It turns out that a squarefree monomial ideal is tame, if and only if the facets of its Stanley-Reisner complex have mutually disjoint complements. Also, we characterize all monomial ideals generated in degree at most 2 which are tame. Finally, we prove that tame squarefree ideals are of fiber type.
\end{abstract}
\maketitle


\def\theenumi{\roman{enumi}}

\section*{Introduction}
The blowup of a scheme $X$ along  the closed subscheme $Z\subset X$ is the scheme $\mathrm{Bl}_Z (X)$ together
with a proper rational map $\pi:\mathrm{Bl}_Z (X)\rar X$ which is an isomorphism outside of $Z$. The closed
subscheme $Z$ is called the center of the blowup. It is well known that the blowup of a regular scheme in a
regular center is regular. A natural question which arises form this statement  is: ''when the blowup of a regular
scheme in non-regular center remains regular".

The main  interest in this article is the question of determining  squarefree monomial ideals such that the blowup
of the affine space along these ideals is regular. This work is inspired by a paper of E. Faber and D.B. Westra
\cite{FW}. They give a  smoothness criterion, based on convex geometry, for a monomial ideal such that  the
blowing up $\mathbb{A}_k^n$ along this ideal is regular. These  ideals have been dubbed  \textit{tame}. They also
find  some classes of tame  monomial ideals such as monomial building sets and permutohedra.

Throughout this paper, the base field $k$ will be considered as an algebraically closed field. Let $R=k[x,y,z]$ be a polynomial ring and let
$Z$ stand for the closed subscheme of affine space $\mathbb{A}_k^3=\spec R $ defined by the monomial ideal
$I=(x,yz)$. The blowup of $\mathbb{A}_k^3$ along $Z$ (or with center $Z$) is singular. In fact, the affine chart
corresponding with the ring extension $R\injects R\left[\frac{yz}{x}\right]\simeq k[x,y,z,t]/(yz-xt)$, defines a
singular variety. Assume that $X\subset \mathbb{A}_k^3$ is  an affine scheme such that the singular subscheme of
$X$ is defined by $I=(x,yz)$ (see e.g. \cite{FW}). If we want to resolve $X$ by one blowup, we have to use the
singular subscheme as a center. Then the blowup of $X$ is embedded in a singular ambient scheme, because $I$ is
not tame. In this case,  we can not speak of an embedded resolution of singularity of $X$ \cite{Cutk}  (see also
\cite{Hauser1}). But note that the blowup of $\mathbb{A}_k^3$ in  a center  defined by the ideal
$(x,yz)(x,y)(x,z)$ is regular and the latter ideal is tame \cite{Hauser2}.

The blowup of an affine space $\mathbb{A}_k^n$ in a center defined by a monomial ideal in $k[x_1,\ldots,x_n]$ is a
toric variety. Therefore, we may  restate the tameness property in combinatorics. In this paper, we determine tame
monomial ideals in $k[x_1,\ldots,x_n]$ which correspond to  clutters and graphs with loops.

The outline of the paper is as follows. In section 1,  we state definition of blowup via the Rees algebra. For a complete discussion and introduction
to blowup and resolution of singularity we refer to \cite{lipman} and \cite{Orlando}.  We focus on blowup of
$\mathbb{A}_k^n$ along monomial ideals and recall some definitions in convex geometry. We give an algebraic
description of the smoothness criterion presented in \cite[Theorem 12(ii)]{FW}. 

In section 2, a full combinatorial characterization of tame squarefree monomial ideals is given. A clutter $\cc$
is called tame if the circuit ideal $I(\cc)$ is tame. One of the main results of this paper is that $\cc$ is tame
if and only if $\cc$ 
 is a union of some isolated vertices and a complete $d$-partite $d$-uniform clutter (Theorem~\ref{main cluts}). It turns out that, a squarefree monomial ideal is tame if and only if the facets of the Stanley-Reisner complex of $I(\cc)$ have mutually disjoint complements (Proposition~\ref{nice criterion for tameness}). Also in this section, it is shown that if the polarization of a monomial ideal $I$ is tame then $I$ is tame, but the converse is not true in general.

In section 3, we present a characterization of  tame monomial ideals generated in degree at most two. In particular, if
$G$ is a graph without isolated vertices (possibly with loops) then the edge ideal $I(G)$ is tame if and only if
$G$ is a looped star, looped complete or simple complete bipartite graph (Corollary \ref{simple graphs} and Theorem
\ref{main looped graph}).

Finally, we give an explicit description of the defining ideal of affine charts $U_i$ of blowup
$\proj{\mathcal{R}_R(I)}$ where $I$ is the circuit ideal of a complete $d$-partite $d$-uniform clutter. By using
this, we find the defining equations of the Rees algebra of a tame squarefree monomial ideal. In particular, it is
proved that the circuit ideal of a tame squarefree monomial ideal is of fiber type.


\section{Blowup along monomial ideals}
Let $R$ be a Noetherian ring and $I\subset R$ be an ideal. The \textit{Rees algebra} of $I$ is defined to be the
graded algebra  $\mathcal{R}_R(I)=R[It]=\oplus_{i\geq 0} I^it^i\subset R[t]$. Assume $I$ is generated by
$f_1,\ldots,f_m\in R$. Consider the polynomial ring $S=R[T_1,\ldots,T_m]$, where $T_i$ are indeterminates. Then
there is a natural ring homomorphism $\varphi: S\surjects \mathcal{R}_R(I) $ that sends $T_i$ to $f_it$. Let
$\mathcal{J}=\ker \varphi$ be the defining ideal  of $\mathcal{R}_R(I)$. Then $\mathcal{R}_R(I)\simeq
S/\mathcal{J}$ and $\mathcal{J}=\bigoplus_{i=1}^{\infty}\mathcal{J}$ is a graded ideal. A minimal generating set
for $\mathcal{J}$ is called the \emph{defining equation set} of the Rees algebra. Also, $\mathcal{J}_1 $ is known
as the defining ideal of the symmetric algebra of $I$, in fact $\mathcal{J}_1= I_1({\mathbf T}.\psi)$ is the ideal
generated by one minors of the product of the variable matrix $\mathbf{T}=[T_1\ T_2\ \ldots \ \ T_n]$  by the
first syzygy matrix $\psi$ of $I$.

Let $(R,\mathfrak{m})$ be a Noetherian local ring  and $I\subset R$ an ideal, the \textit{special fiber} of $I$ is
defined to be $\mathcal{F}(I)={\rm gr}_R(I)\otimes R/\mathfrak{m}$, where ${\rm
gr}_R(I)=\mathcal{R}_R(I)/I\mathcal{R}_R(I)=\bigoplus_{i=0}^{\infty}I^i/I^{i+1}$. In the case that $R$ is a
polynomial ring over a field $k$ and $I=(f_1,\ldots,f_m)$ the special fiber $\mathcal{F}(I)$ is isomorphic to
$k[f_1,\ldots, f_m]$. Then there is a homomorphism $\varPsi: k[T_1,\ldots,T_m]\surjects \mathcal{F}(I)$ that maps
$T_i$ to $f_i$. Set $\mathcal{H}=\ker \varPsi$. The ideal $I$ is called of\textit{ fiber type} if
$\mathcal{J}=S\mathcal{J}_1+S\mathcal{H}$.

Suppose that $X=\spec{R}$ is an affine scheme and let $Z=\spec{R/I}$ be a closed subscheme of $X$ defined by an
ideal $I$ of $R$. The \textit{blowup} of $X$ along  $Z$ is the scheme ${\mathrm{Bl}_Z \left(X \right)}={\rm Proj}(\mathcal{R}_R(I))$
together with the morphism $\pi\colon {\mathrm{Bl}_Z \left(X \right)}\rar X$ given by the natural ring homomorphism $R\rar
\mathcal{R}_R(I)$. The subscheme $E=\pi^{-1}(Z)$ of ${\mathrm{Bl}_Z \left(X \right)}$ is called the \textit{exceptional divisor}
of the blowup. The blowup ${\mathrm{Bl}_Z \left(X \right)}$ can be embedded into $\mathbb{P}^{m-1}_R$ as the closed subscheme
defined by the defining ideal $\mathcal{J}$ of the Rees algebra of $I$, where $m$ is the size of a generating set
for $I$. The following statement can be found in standard textbooks, like  \cite{Cutk}.
\begin{Proposition}\label{charts}
Let $X=\spec{R}$ be an affine scheme and let $Z$ be a closed subscheme defined by an ideal $I=(f_1,\ldots, f_m)$
of $R$. The following statements hold:
\begin{enumerate}
\item  The blowup  of $X$ along $Z$ can be covered by  affine charts 
$$\spec{R\left[ \frac{f_1}{f_i},\ldots, \frac{f_m}{f_i} \right]},$$
for $i=1,\ldots,m$.
\item  The defining ideal of $\mathrm{Bl}_Z \left(X \right)$ in the  $i$'th affine chart $\spec{R[f_1/f_i,\ldots, f_m/f_i]}$ is
    given by dehomogenization of the defining ideal $\mathcal{J}$ of the Rees algebra of $I$  with respect to
    variable $T_i$.
\end{enumerate}
\end{Proposition}

\begin{proof}
For $f\in I$, let $R[If^{-1}]$ be the subalgebra of $R_f$ which is generated by elements of the form $x/f^d$ with
$x\in I, d\in \n$. Such an element may be also considered as a degree zero element of $\mathcal{R}_R(I)_f$. In
fact, there is an $R$-algebra isomorphism  $R[If^{-1}]\simeq \mathcal{R}_R(I)_{(f)}$. The inverse map of this
isomorphism is given by considering $y/f^i$ for $y\in I^i$ as an element of $R[If^{-1}]$. We can see that if $f$
runs through a generating set of $I$, then we have
\[
\mathrm{Bl}_Z X=\bigcup_{i=1}^m\spec{\left[ {I}/{f^i} \right]}=\bigcup_{i=1}^m \spec{R \left[ \frac{f_j}{f_i}\colon \; 1\leq j\leq m \right]}. 
\]

Let $\widetilde{\mathcal{J}}$ stand for the defining ideal of the affine chart $R[f_2/f_1,\ldots,f_m/f_1]$, then
\[
\frac{R \left[ \frac{T_2}{T_1},\ldots,\frac{T_m}{T_1} \right]}{\widetilde{\mathcal{J}}} \simeq R \left[ \frac{f_2}{f_1},\cdots,\frac{f_m}{f_1} \right].
\]
Assume that $G_1(T_1,\ldots,T_m),\ldots, G_s(T_1,\ldots,T_m)$ with $\deg G_i=d_i$ are the defining equations of
the Rees algebra of $I$. We claim that $\widetilde{\mathcal{J}}=(g_1,\ldots, g_s)$ where
 $$g_i\left( \frac{T_2}{T_1},\ldots,\frac{T_m}{T_1} \right)=T_1^{-d_i}G_i.$$
Note that $g_i$ is the dehomogenization of $G_i$ with respect to the variable $T_1$. We first prove that $g_i\in
\widetilde{\mathcal{J}}$. It suffices to prove that $f_1^{d_i}g_i(f_2/f_1,\ldots,f_m/f_1)=0$. However
$f_1^{d_i}g_i(f_2/f_1,\ldots,f_m/f_1)=G_i(f_1,\ldots,f_m)=0$.

Now let  $g\in \widetilde{\mathcal{J}}$. Setting $d=\deg g$. Denote by $G$ the homogenization of $g$ with respect
to $T_1$, that is $G(T_1,\ldots,T_m)=T_1^dg(T_2/T_1,\ldots,T_m/T_1)$. Then
\[
G(f_1,\ldots,f_m)=f_1^d \ g\left( \frac{f_2}{f_1},\ldots,\frac{f_m}{f_1} \right)=0.\]
So that $G(T_1,\ldots, T_m)$ is a homogeneous polynomial in $\mathcal{J}$. Thus there are $H_1,\ldots, H_s$  with
$\deg H_i=d-d_i$ in $R[T_1,\ldots,T_m]$ such that $G=\sum_{i=1}^{s}G_iH_i$. Hence
\[
g=T_1^{-d}G(T_1,\ldots,T_m)=\sum_{i=1}^{s}T_1^{d_i-d}H_i(T_1,\ldots,T_m)\ g_i\left( \frac{T_2}{T_1},\ldots,\frac{T_m}{T_1} \right)
\]
and then letting $h_i(T_2/T_1,\ldots,T_m/T_1)=T_1^{d_i-d}H_i(T_1,\ldots,T_m)$, one has $g=\sum_{i=1}^{s}g_ih_i$.
\end{proof}

Assume that $\mathbb{A}_k^n=\spec{k[x_1,\ldots,x_n]}$ and let $Z$ be the subscheme defined by the ideal of
coordinate $I=(x_1,\ldots,x_m)$. Then, for $i=1,\ldots,m$
\[
k \left[ \mathbf{x} \right] \left[ \frac{x_1}{x_i},\ldots, \frac{x_m}{x_i} \right] \simeq
 \frac{k \left[ \mathbf{x} \right] \left[ T_1,\ldots,T_{i-1},T_{i+1},\ldots, T_m \right]}
 { \left( x_j-x_iT_j \colon \; j \neq i, \; j \leq m \right)}.
\]
The latter ring is isomorphic with the polynomial ring with $n$ variables. Therefore, we can see that in each
affine chart, the blowup $\mathrm{Bl}_Z \left( \mathbb{A}_k^n \right)$ is a regular affine variety. More generally, if the center
$Z$ given by the ideal $I$ is regular,  $I$ is generated by a regular system of parameters say $f_1,\ldots, f_m$,
then the defining ideal of the Rees algebra of $I$ is generated by $2\times 2$-minors of the matrix
\[
\begin{bmatrix}
	f_1 & f_2 & \ldots & f_m\\
	T_1 & T_2 & \ldots & T_m
\end{bmatrix}.
\]
Thus, like above in each affine chart $\mathrm{Bl}_Z \left( \mathbb{A}_k^n \right)$ is a polynomial ring and hence regular.

In general, as in the example in the introduction, blowing up  $\mathbb{A}_k^n$ in a non-regular subscheme may
produce singularity. One class of ideals which define non-regular schemes is the class of monomial ideals which
have at least one generator of degree at least 2. In this paper, we focus on such ideals. In the sequel, we set
$R=k[x_1,\ldots,x_n]$ and denote the blowup of $\mathbb{A}_k^n$ along a subscheme defined by a monomial ideal $I$
by $\mathrm{Bl}_I (R)$.

For an starting point to characterize squarefree tame monomial ideals, we need some results of \cite{FW} which
uses convex geometry. Set $\supp I= \{\textbf{a}\in \n^n \colon \; \textbf{x}^\textbf{a} \in I\}$ (here $\n$
denotes the set of non-negative integers), where $I$ is a monomial ideal, and let $N(I)$ be the convex hull of
$\supp I$. A point $\textbf{p}$ in $N(I)$ is said to be a \emph{vertex} when $\textbf{p}$ cannot be written as
$\lambda_1 \textbf{p}_1 + \lambda_2 \textbf{p}_2$, for some $0<\lambda_i<1$ with $\lambda_1+\lambda_2=1$ and
$\textbf{p}_1 \neq \textbf{p}_2\in N(I)$. Suppose that $\textbf{x}^\textbf{a} \in I$ for some $\textbf{a} \in
\n^n$. We say that $\textbf{x}^\textbf{a}$ is a vertex of $I$, when $\textbf{a}$ is a vertex of $N(I)$.

Let $\textbf{a}\in \supp I$, $u=\textbf{x}^\textbf{a}$ and $G(I)$ be the minimal set of generators of $I$. Then by
the $\textbf{a}$-chart (or as sometimes we call it, the $u$-chart) of the blowup $\Bl{I}$ of $\bbA_k^n$ along $I$,
we mean $\spec{k[U]}$ where $U=\{x_1, \ldots, x_n\}\cup\{\f{u'}{u}\colon\; u\neq u'\in G(I)\}$. Thus by Proposition~\ref{charts},
$\Bl{I}$ is covered by the $u$-charts ($u\in G(I)$). Faber and Westra presented a method to check the tameness of
a monomial ideal $I$ of $R$ in \cite[Theorem 12]{FW} in terms of convex geometry. Here we restate their result in
a more algebraic language.
\begin{prop}\label{FW crit}
Suppose that $u$ is a vertex of the monomial ideal $I$ of $R$. Then,
\begin{enumerate}
\item the $u$-chart of $\mathrm{Bl}_I (R)$ is
regular \ifof there is a $U'\se U$, with $|U'|=n$ such that $k[U]=k[U']$, where $U=\{x_1, \ldots, x_n\}\cup\{\f{u'}{u}\colon\; u\neq u'\in G(I)\}$.
\item $I$ is tame \ifof every chart of $\Bl{I}$ which correspond to a vertex of $I$ is regular.
\end{enumerate}
\end{prop}

\begin{proof}
(i) Since $u$ is a vertex of $I$, it follows \cite[Theorem 12(ii)]{FW} that the $u$-chart ($=\spec{k[U]}$) is regular
\ifof the minimal set of monomials $U'\se S$ such that $U'$ generates $k[U]$ as a $k$-algebra, has exactly
cardinality $n$. Note that this set always has at least $n$ elements by \cite[Lemma 9]{FW}. Also this minimal set
is unique by Lemmata 6 and 7 of \cite{FW}. Consequently, we can obtain $U'$ from $U$ by deleting monomials $u\in
U$ which can be written as a product of monomials in $U$ different from $u$. In particular, $U'\se U$ as required.

Statement (ii) follows from \cite[Theorem 12(i)]{FW}.
\end{proof}
%
%


\section{Tame Clutters} \label{Section Tame clutters}
In this section, we present a full combinatorial characterization of tame squarefree monomial ideals. Recall that
a \emph{clutter} $\cc$ on vertex set $[n]=\{1,\ldots, n\}$, is a family of incomparable subsets of $[n]$ (ordered
by inclusion). Throughout this paper, $\cc$ denotes a clutter on $[n]$. The elements of $\cc$ are called\emph{
circuits} of $\cc$ and $\cc$ is called\emph{ $d$-uniform} when all of its circuits have cardinality $d$. Let $F\se
[n]$. By $\textbf{x}_F$ we mean $\prod_{i\in F} x_i$ in the polynomial ring $R$. Also the \emph{circuit ideal} of
$\cc$ is $I \left( \cc \right)= \left( \textbf{x}_F \colon \; F\in \cc \right)$. Thus there is a one-to-one
correspondence between squarefree monomial ideals of $R$ and clutters on $[n]$. We say that $\cc$ is \textit{tame}
when $I(\cc)$ is so. Moreover, for a subset $A \se [n]$ we call the set $\N(A)= \N_\cc(A)= \l\{v\in [n]\sm A \colon
\; \exists e\in \cc \text{ s.t. } A \cup\{v\}\se e \r\}$, the \emph{open neighborhood} of $f$ in $\cc$.

Let $d'\geq d$ be two positive integers. Following \cite{d-partite}, we say that a $d$-uniform clutter $\cc$ is
\emph{$d'$-partite}, if the set of vertices can be written as the union of mutually disjoint subsets $V_1, \ldots,
V_{d'}$, such that each circuit of $\cc$ meets each $V_i$ in at most one vertex. If moreover, $\cc$ contains all
$d$-subsets of $[n]$ which intersect each $V_i$ in at most one vertex, we say that $\cc$ is \emph{complete
$d'$-partite}. The partition $\{V_i\colon \; i \in [d'] \}$ as above is called a \emph{$d'$-partition} of $\cc$.
Here we prove that $\cc$ is tame \ifof $\cc$ a complete $d$-partite $d$-uniform clutter.

Not all the minimal generators of a monomial ideal $I$ need be vertices. For example if $I=(x_1^2,x_2^2,x_1x_2)$,
then obviously $x_1x_2$ is not a vertex of $I$. But as the following lemma shows, in the squarefree case, the
vertices of $I$ are exactly the minimal generators of $I$.

\begin{lem}\label{vertex=gen}
Suppose that $I$ is a squarefree monomial ideal of $R$ with minimal generating set $\mathcal{G}(I)=\{ \textbf{\rm{\textbf{x}}}^{\textbf{\rm{\textbf{a}}}_i}\colon \; i\in [s] \}$ and let $\textbf{\rm{\textbf{a}}} \in N(I)$. Then $\textbf{\rm{\textbf{a}}}$ is a vertex of $N(I)$ \ifof $\textbf{\rm{\textbf{x}}}^\textbf{\rm{\textbf{a}}} \in \mathcal{G}(I)$.
\end{lem}

\begin{proof}
(\give): Clearly $\textbf{a}\in \supp I$. If $\textbf{x}^{\textbf{a}} \notin \mathcal{G}(I)$, then $\textbf{a}=\textbf{a}_i+ \textbf{b}$, for some $i\in [s]$ and some non-zero vector
$\textbf{b} \in \n^n$. Assume that $b_1>0$. Then $\textbf{a} =\f{1}{2}(\textbf{a}_i+ \textbf{b} + \textbf{e}_1)+\f{1}{2}(\textbf{a}_i + \textbf{b} -\textbf{e}_1)$ where $\textbf{e}_i$'s are the standard basis of $\z^n$. Thus $\textbf{a}$ is not a vertex, a contradiction.

(\rgive): We show that for example $\textbf{a}_1$ is a vertex of $N(I)$. On the contrary, assume that
$\textbf{a}_1=\lambda_1 \textbf{p}_1+\lambda_2 \textbf{p}_2$, for some $0<\lambda_i<1$ with
$\lambda_1+\lambda_2=1$ and $\textbf{p}_1\neq \textbf{p}_2\in N(I)$. Replacing $\textbf{p}_i$'s with their convex
combinations of elements of $\supp I$, it follows that $\textbf{a}_1=\sum_{j=1}^t \lambda_j (\textbf{a}_{i_j}
+\textbf{a}'_j)$ where $t\geq 2$,  $0 < \lambda_j<1$, $i_j \in [s]$ and $\textbf{a}'_j \in \n^n$ for all $j$,
$\sum_{j=1}^t \lambda_j=1$ and the points $\textbf{p}_j=(\textbf{a}_{i_j} +\textbf{a}'_j)$ are mutually distinct.

If for all $j$, we have $i_j=1$, then it easily follows that $\textbf{a}'_j=0$ for all $j$, contradicting the
mutually distinctness of $\textbf{p}_j$'s. Thus there exists $j\in [t]$, with $i_j \neq 1$, say $j=1$ and $i_1=2$.
There is $d\in \n$ such that $d\lambda_1>1$. Then for each $i\in [n]$, the $i$'th component of $d\textbf{a}_1$ is
at least equal to the $i$'th component of $\textbf{a}_2$. So $(\textbf{x}^{\textbf{a}_1})^d\in
R\textbf{x}^{\textbf{a}_2}$ and $\textbf{x}^{\textbf{a}_1} \in \sqrt{R\textbf{x}^{\textbf{a}_2}}=
R\textbf{x}^{\textbf{a}_2}$, for $\textbf{x}^{\textbf{a}_2}$ is squarefree. But this is in contradiction with the
minimality of the generating set $\mathcal{G}(I)$ and the result is concluded.
\end{proof}

When $e_0\in \cc$, by the $e_0$-chart of $\Bl{I(\cc)}$ we mean the $\textbf{x}_{e_0}$-chart, that is,
$\spec{k[U]}$ where $U$ is as in the notes above Lemma~\ref{FW crit} with $u=\textbf{x}_{e_0}$. An immediate
consequence of Lemmata~\ref{FW crit} and \ref{vertex=gen} is the following.
\begin{cor}\label{tame crit for sq frees}
Assume that $e_0$ and $U$ are as above and $S=k[U]$. Then $\spec S$ is regular \ifof there is a subset $U'\se U$
with $|U'|=n$ such that $S=k[U']$.
\end{cor}

Now we can apply Corollary~\ref{tame crit for sq frees} to find a combinatorial characterization of tame clutters. But first we need some lemmas.

\begin{lem}\label{gen in general cluts}
Let $e_0\in \cc$, $U_1=\{x_1,\ldots, x_n\}$, $U_2=\{\f{\textbf{\rm \textbf{x}}_e}{\textbf{\rm \textbf{x}}_{e_0}}
\colon \; e_0 \neq e \in \cc \}$ and $U=U_1 \cup U_2$. Suppose that the $e_0$-chart of $\Bl{I(\cc)}$ is regular
and $U'$ is as in  Corollary~\ref{tame crit for sq frees}. Then,
\begin{enumerate}
\item for each $x_j \in U_1 \sm U'$, there exists a nonempty set $\pi(j) \se e_0$ such that  $U' = \left( U_1
    \cap U' \right) \cup \{ \f{x_j}{\textbf{\rm{\textbf{x}}}_{\pi(j)}} \colon \; x_j \in U_1 \sm U'\}$;
\item for each $j \in e_0$, one has $x_j\in  U'$.
\end{enumerate}
\end{lem}

\begin{proof}
(i): It is sufficient to show that for each $x_j\in U_1\sm U'$, there is a $\pi(j)\se e_0$ such that
$\f{x_j}{\textbf{x}_{\pi(j)}}\in U'$, since then $U''= \left( U_1 \cap U' \right) \cup
\{\f{x_j}{\textbf{x}_{\pi(j)}}\colon \;x_j \in U_1 \sm U' \} \se U'$ and $|U''|=|U_1|=n=|U'|$ (by
Corollary~\ref{tame crit for sq frees}) and the result follows.

Assume that $x_j \notin U'$ for some $j \in [n]$. Then as $k[U]=k[U']$, we can write $x_j$ as a product
\begin{equation} \label{star}
x_j= u_1u_2\cdots u_t x_{i_1} \cdots x_{i_r},
\end{equation}
where $u_i=\f{\textbf{x}_{e_i}}{\textbf{x}_{e_0}} \in U'\cap U_2$. If $x_s|\textbf{x}_{e_i}$,
then we should have $s \in e_0\cup \{j\}$, else $x_s$ does not cancel out and should appear in the left hand side of (\ref{star}), which is not the case. A similar argument shows that at most one of the $e_i$'s contain $j$, and since none of the $e_i$'s is contained in $e_0$, we get $t\leq 1$. Clearly $t\geq 1$, hence $t=1$ and $u_1=\f{x_j}{x_{i_1}\cdots x_{i_r}}$. Set $\pi(j)= \{i_1, \ldots, i_r\}$ which is clearly nonempty. Since $u_1\in U'\se U$, we deduce that $\textbf{x}_{\pi(j)}|\textbf{x}_{e_0}$ which means $\pi(j)\se e_0$, as claimed.

(ii): Note that, if $j\in e_0$ and $x_j\notin U'$, then in (\ref{star}) all $e_i$'s should be contained in $e_0$
which is not possible (since $\cc$ is a clutter). Thus for $j \in e_0$, $x_j\in U'$.
\end{proof}

In what follows, when we say that a monomial $u$ of $R= k[x_1,\ldots, x_n]$ divides the numerator [resp.
denominator] of $q=\f{p_1}{p_2}$ where $p_1$ and $p_2$ are monomials of $R$, we mean that $u$ divides the
numerator [resp. denominator] of $q$ when $q$ is written in the simplest form. Also we consider $k(x_1, \ldots,
x_n)$ in its standard grading, that is, $\deg q=\deg p_1-\deg p_2$. Furthermore, an \emph{isolated vertex} of a
clutter $\cc$ means a vertex which has not appeared in any circuit of $\cc$.

\begin{lem}\label{gen in uniform cluts}
Suppose that the conditions of Lemma~\ref{gen in general cluts} holds and for simplicity assume $e_0=[d]$. Moreover, assume that $\cc$ is $d$-uniform and without any isolated vertex. Then for each $d<j\leq n$, there exists a $v(j) \leq d$ such that $U'= \{x_1, \ldots, x_d\}\cup \{\f{x_j}{x_{v(j)}}\colon \; d<j\leq n\}$.
\end{lem}

\begin{proof}
We use the notations of Lemma~\ref{gen in general cluts}. First note that for each $x_j \in U_1 \sm U'$, we have $u=\f{x_j}{\textbf{x}_{\pi(j)}}\in U'\se U$. Since $\cc$ is uniform, we conclude that the degree of every element of $U$, including $u$, is non-negative. Consequently, $\deg u=0$ and $|\pi(j)|=1$, say $\pi(j)=\{v(j)\}$.

Therefore, according to Lemma~\ref{gen in general cluts}, we just need to prove that if $j>d$, then $x_j \notin U'$. On the contrary suppose that $x_j\in U'$ for some $j>d$. As $j$ is not an isolated vertex of $\cc$, there is an $e\in \cc$ with $j\in e$. Thus $u=\f{\textbf{x}_e}{\textbf{x}_{e_0}}\in U$ should be a product of monomials in $U'$, say $u=u_1\cdots
u_t$, ($u_i\in U'$). Since $j\notin e_0$, $x_j$ divides the numerator of $u$. But by Lemma~\ref{gen in general cluts}, the only monomial in $U'$  with numerator divisible by $x_j$, is $x_j$ itself. Hence we can assume that $u_1=x_j$.
If $t=1$ then $e=e_0\cup \{j\}$ which contradicts $\cc$ being a clutter. So $t>1$ and we deduce that $u'=
\f{\textbf{x}_{e\sm \{j\}}}{\textbf{x}_{e_0}}= u_2\cdots u_t$. Now all $u_i$'s have degree $\geq 0$ but $\deg u'=-1$ (because $\cc$ is uniform), a contradiction from which the result follows.
\end{proof}

\begin{thm}\label{regular chart}
Let $\cc$ be a $d$-uniform clutter without isolated vertices and $e_0=[d]\in \cc$. The $e_0$-chart of
$\Bl{I(\cc)}$ is regular \ifof $\cc$ is a $d$-partite clutter with $d$-partition $\l\{ \N(e_0\sm \{i\}) \colon \;
i\in [d] \r \}$.
\end{thm}
\begin{proof}
(\give): Assume that $U$, $U'$ and $v(j)$'s are as in Lemma~\ref{gen in uniform cluts}. Also for $i\leq d$ set $v(i)=i$.
Let $V_i= \{j\in [n]\colon \; v(j)=i\}$ for each $i\in [d]$. Clearly $V_i$'s form a partition of $[n]$. Suppose that for
some $e\in \cc$ and $i\in [d]$, $|e\cap V_i|>1$, say $r\neq s\in e\cap V_i$. Because $e\neq e_0$ and $u=
\f{\textbf{x}_e}{\textbf{x}_{e_0}} \in U$ and since $\deg u=0$, we deduce that $u$ is a product of monomials in $U'$ with zero
degree. That is, $u=\f{x_{j_1}}{x_{v(j_1)}} \cdots \f{x_{j_d}}{x_{v(j_d)}}$ for some $j_l\in [n]$ (note that we
are using the convention $v(j)=j$ for $j\leq d$). Since $x_r,x_s|x_e$, we can assume that $j_1=r$ and $j_2=s$. So
$x_{v(r)}x_{v(s)}=x_i^2|x_{e_0}$, a contradiction.

From this contradiction, we conclude that $V_i$'s form a $d$-partition of the $d$-uniform clutter $\cc$. Whence
$\N(e\sm \{j\})$ is contained in $V_{v(j)}$ for each $j\in e\in \cc$. In particular, $\N(e_0\sm \{i\})\se V_i$ for
each $i\in [d]$. Now as $\f{x_j}{x_{v(j)}}\in U$, there should exist $e\in \cc$ with $\f{\textbf{x}_e}{\textbf{x}_{e_0}}=
\f{x_j}{x_{v(j)}}$, that is, $e=(e_0\sm \{v(j)\}) \cup \{j\}\in \cc$. Therefore, $j\in \N(e_0\sm \{v(j)\})$ and
hence the union of $\N(e_0\sm \{i \})$'s ($i\in [d]$) is the whole $[n]$. It follows that $V_i= N(e_0\sm \{i\})$,
as required.

(\rgive): Let $V_i= \N(e_0\sm \{i \})$. Thus $e= (e_0\sm \{i\}) \cup \{j\} \in \cc$ for each $j\in V_i$. So if we
set $v(j)=i$, then $\f{x_j}{x_{v(j)}}= \f{\textbf{x}_e}{\textbf{x}_{e_0}}\in U$. Let $U'=\{x_1,\ldots, x_d\}\cup \{
\f{x_j}{x_{v(j)}} \colon \;  d<j\leq n\}$. Obviously $x_i\in k[U']$ for each $i\in [n]$. If $e'\in \cc$, then $e'=\{j_1,
\ldots, j_d\}$ with $j_i\in V_i$, because $\cc$ is $d$-partite with $d$-partition $V_i$'s. Hence
$\f{\textbf{x}_{e'}}{\textbf{x}_{e_0}}= \f{x_{j_1}}{x_{v(j_1)}} \cdots \f{x_{j_d}}{x_{v(j_d)}} \in k[U']$. Consequently,
$k[U]=k[U']$ and the result follows from Lemma~\ref{tame crit for sq frees}.
\end{proof}

\begin{cor}\label{tame uniform cluts}
Assume that $\cc$ is a $d$-uniform clutter without isolated vertices. Then $\cc$ is tame \ifof $\cc$ is complete
$d$-partite.
\end{cor}

\begin{proof}
(\rgive): Immediate consequence of Theorem~\ref{regular chart}.

(\give): Let $e=[d]\in \cc$. Then by Theorem~\ref{regular chart}, $\cc$ is a $d$-partite clutter with
$d$-partition $ \l\{ V_i= \N(e\sm \{i\}) |i\in [d] \r\}$. First note that the $d$-partition of $\cc$ is unique,
because clearly vertices $1, \ldots, d$ should be in different partitions and each $v\in \N(e\sm \{i\})$ should be
in the same partition as $i$ ($i\in [d]$).

Now assume that there is an $e'\se [n]\sm \cc$, $|e'|=d$ and $|e'\cap V_i|=1$ for each $i\in [d]$. Choose an
$e_0\in \cc$ such that $|e'\cap e_0|$ is the maximum possible. If $a=e'\cap e_0$, then we can assume that
$e'=a\cup \{i_1, \ldots, i_r\}$ and $e_0= a\cup \{j_1,\ldots, j_r\}$ for some $i_l \neq j_l\in V_l$. Since the
$e_0$-chart of $\Bl{I(\cc)}$ is regular it follows Theorem~\ref{regular chart} that $i_1\in V_1= \N(e_0\sm
\{j_1\})$ and so $e'_0= e_0\sm \{j_1\} \cup \{i_1\}\in \cc$. But $|e'\cap e'_0|> |e'\cap e_0|$ contradicting the
choice of $e_0$. Consequently no $e'$ with the above properties exists, that is, $\cc$ is complete $d$-partite.
\end{proof}


Now that we have a characterization of tame uniform clutters, let's pay attention to non-uniform clutters.

\begin{lem}\label{min size}
Let $\cc$ be a clutter and $e_0=[d]\in \cc$. Also assume that $d\leq |e|$ for each $e\in \cc$ and set $\cc'=\{e\in
\cc\colon\; |e|=d\}$.
\begin{enumerate}
\item If the $e_0$-chart of $\Bl{I(\cc)}$ is regular, then the $e_0$-chart of $\Bl{I(\cc')}$ is regular.
\item If $e_0$-chart of $\Bl{I(\cc')}$ is regular, $\cc'$ has no isolated vertices and $U'$ is as in
    Corollary~\ref{tame crit for sq frees}, then for each $d<j\leq n$, there is a $v(j)\leq d$ such that $U'=
    \{x_1, \ldots, x_d\}\cup \{\f{x_j}{x_{v(j)}}\colon \; d<j\leq n\}$.
\end{enumerate}
\end{lem}

\begin{proof}
Suppose that $U$, $U'$ and $S$ are as in Lemma~\ref{tame crit for sq frees} and the notes before it. Also let
$U''=\{x_1, \ldots, x_n\} \cup \{\f{\textbf{x}_e}{\textbf{x}_{e_0}}| e_0 \neq e \in \cc'\}$. Thus the $e_0$-chart
of $\cc'$ is $\spec{k[U'']}$. We just need to show that $U'\se U''$. According to Lemma~\ref{gen in general
cluts}, the monomials in $U'$ which are not a variable are of the form $u= \f{x_j}{x_{\pi(j)}}$, for some
$\pi(j)\se [d]$. So there is an $e\in \cc$ with $u= \f{\textbf{x}_e}{\textbf{x}_{e_0}}$, which means $e=(e_0\sm
\pi(j)) \cup \{j\}$. Since $|e|\geq |e_0|$ we conclude that $|\pi(j)|=1$ and hence $|e|=|e_0|$, that is $e\in
\cc'$ and $u\in U''$. The final statement follows by applying Lemma~\ref{gen in uniform cluts} on the $e_0$-chart
of $\cc'$.
\end{proof}

Now we can present a full characterization of tame clutters.

\begin{thm}\label{main cluts}
Suppose that $\cc$ is a clutter. Then $\cc$ is tame \ifof $\cc$ is a union of some isolated vertices and a
complete $d$-partite $d$-uniform clutter, for some positive integer $d$.
\end{thm}

\begin{proof}
(\rgive): Proved in Corollary~\ref{tame uniform cluts}. (Note that adding or removing isolated vertices does not affect
tameness.)

(\give): Let $d= \min \{ |e| \colon \; e\in \cc \}$ and $\cc'=\{e\in \cc\colon \;  |e|=d\}$. Then by
Lemma~\ref{min size}, $\cc'$ is tame. We may assume that $m \leq n$ is such that the isolated vertices of $\cc'$
are $m+1, \ldots, n$. Applying Corollary~\ref{tame uniform cluts} on $\cc'|_{[m]}$ (that is, viewing $\cc'$ as a
clutter on vertex set $[m]$), we get that $\cc'|_{[m]}$ is a complete $d$-partite $d$-uniform clutter. Let $e\in
\cc$ and $e\neq e_0=[d]\in \cc'$. Then it follows from Lemma~\ref{min size} that, the coordinate ring of the
$e_0$-chart of $\cc'$ is generated over $k$ by $U'= \{x_1, \ldots, x_d\}\cup \{\f{x_j}{x_{v(j)}}\colon \; d<j\leq
m\} \cup \{x_{m+1}, \ldots, x_n\}$ where each $v(j)\leq d$. In particular,
$$u= \f{\textbf{x}_e}{\textbf{x}_{e_0}}= \f{x_{j_1}}{x_{v(j_1)}} \cdots \f{x_{j_t}}{x_{v(j_t)}} x_{i_1}\cdots x_{i_r},$$
where $d<j_1, \ldots, j_t\leq m$. As for each $i\in [d]\sm e$ the denominator of $u$ is divisible by $x_i$, we
deduce that such an $i$ should appear as some $v(j_l)$, say $v(j_i)$. It follows that $j_i\in \N_{\cc'}(e_0\sm
\{i\})$ which is the $i$'th partition of $\cc'|_{[m]}$. Since $x_{j_l}$'s divide the numerator of $u$, we have
$e'= \{j_i \colon \; i\in e_0\sm e\} \cup (e\cap e_0)\se e$. But $e'$ meets each partition of $\cc'|_{[m]}$ in
exactly one vertex and so $e'\in \cc'\se \cc$ which is in contradiction with $\cc$ being a clutter, unless $e=e'$.
Therefore, $|e|=d$ and as $e$ was arbitrary we conclude that $\cc=\cc'$ is a union of some isolated vertices and a
complete $d$-partite $d$-uniform clutter.
\end{proof}

Let $F\se [n]$. By $P_F$ we mean the prime ideal of $R$ generated by $x_i$'s with $i\in F$. Using this notation,
the algebraic restatement of the above theorem is:
\begin{cor} \label{tamness and simplicial complexes}
Suppose that $I$ is a proper squarefree monomial ideal of $R$. Then the following are equivalent:
\begin{enumerate}
\item the blowup of $\bbA_k^n$ along $I$ is regular;
\item there exist mutually disjoint nonempty subsets $F_1, \ldots, F_d$ of $[n]$ such that
    $I=P_{F_1}P_{F_2}\cdots P_{F_d}=P_{F_1}\cap P_{F_2}\cap \cdots \cap P_{F_d}$.
\end{enumerate}
\end{cor}


\subsection*{Tameness via Stanley-Reisner complexes}
Let $R=k[x_1, \ldots, x_n]$ be the polynomial ring over a field $k$ and $\mathfrak{m} = \left( x_1, \ldots, x_n
\right)$ be its irredundant (homogeneous) maximal ideal. Squarefree monomial ideals $I \subset \mathfrak{m}^2$ are
in one-to-one correspondence to simplicial complexes on $[n]=\{1, \ldots, n\}$, via Stanley-Reisner ideals. The
interaction between combinatorial behaviour of simplicial complexes and algebraic (geometric) properties of
corresponding ideals (varieties) is wide area of research in combinatorial commutative algebra (algebraic
geometry). In the first part of this section, it was observed that, for a squarefree monomial ideal $I$, the
property of being tame is eventually a combinatorial property. In the following, we state an equivalent condition
on a simplicial complex, such that the corresponding Stanley-Reisner ideal is tame.

\begin{defn}[Simplicial complex]
A \textit{simplicial complex} $\Delta$  over a set of vertices $V=\{ v_{1}, \ldots, v_{n} \}$, is a collection of subsets of $V$, with the property that:
\begin{itemize}
\item[(a)] $\{ v_{i} \} \in \Delta $, for all $i$;
\item[(b)] if $F\in \Delta$, then all subsets of $F$ are also in $\Delta$ (including the empty set).
\end{itemize}
An element of $\Delta$ is called a \textit{face }of $\Delta$ and a \emph{non-face} of $\Delta$ is a subset $F$ of $V$ with $F \notin \Delta$. The maximal faces of $\Delta$ under inclusion are called \textit{facets }of $\Delta$. Let $\mathcal{F}(\Delta) =\{F_{1}, \ldots, F_{q}\}$ be the facet set of $\Delta$. It is clear that $\mathcal{F}(\Delta)$ determines $\Delta$ completely and we write $\Delta = \langle F_{1}, \ldots, F_{q} \rangle$.
\end{defn}

Let $\Delta$ be a simplicial complex over $n$ vertices labelled $v_{1}, \ldots, v_{n}$. The \emph{non-face ideal}
or the \emph{Stanley-Reisner ideal} of $\Delta$, denoted by $I_\Delta$, is the ideal of $R$ generated by
squarefree monomials $\{ \textbf{x}_F \colon \; F \in \mathcal{N}(\Delta) \}$. One may notice that there exists
a one-to-one correspondence between squarefree monomial ideals in $\mathfrak{m}^2$ and Stanley-Reisner ideals of
simplicial complexes. It is well-known that $I_\Delta=\bigcap_{F \in \mathcal{F}(\Delta)}P_{\bar{F}}$ is a prime
decomposition of $I_\Delta$, where $P_{\bar{F}}$ denotes the (prime) ideal generated by all $\{x_i \colon v_i
\notin F \}$ (see e.g \cite{Stan1}, for more details). By virtue of Corollary~\ref{tamness and simplicial
complexes}, we may characterize tame squarefree monomial ideals in terms of combinatorics of the associated
simplicial complexes.

\begin{prop} \label{nice criterion for tameness}
Let $I$ be a squarefree monomial ideal and $\Delta$ be a  simplicial complex on vertex set $[n]$ such that
$I_\Delta =I$. The following are equivalent.
\begin{enumerate}
\item $I$ is a tame ideal.
\item For any two distinct facets $F, G \in \mathcal{F}(\Delta)$, we have $F \cup G =[n]$.
\end{enumerate}
\end{prop}

\begin{proof}
Let $\mathcal{F}(\Delta) = \{ G_1, \ldots, G_s \}$. Then
\begin{equation}\label{Local EQ 3}
I= I_\Delta = P_{\bar{G}_1} \cap \cdots \cap P_{\bar{G}_s}
\end{equation}
is the minimal prime decomposition of $I$.

(i) \give\ (ii): Since $I$ is a tame squarefree monomial ideal, it follows from Corollary~\ref{tamness and
simplicial complexes} that:
$$I = P_{F_1} \cap \cdots \cap P_{F_r},$$
where $F_1, \ldots, F_r$ are mutually disjoint subsets of $[n]$. Since the minimal prime decomposition of a
squarefree monomial ideal is unique (up to a permutation of prime components), we conclude that, for all $i \in
[r]$ there exits a (unique) $j \in [s]$, such that $F_i = \bar{G}_j$. The assertion now follows from the fact that
$F_1, \ldots, F_r$ are mutually disjoint.

(ii) \give (i): Our assumption in (ii) implies that, $\bar{G}_i$'s are mutually disjoint subset of $[n]$. Thus the
desired conclusion follows from (\ref{Local EQ 3}) and  Corollary~\ref{tamness and simplicial complexes}.
\end{proof}

\begin{rem}
The statement in Proposition~\ref{nice criterion for tameness}, provides a simple algorithm for detecting tameness of squarefree monomial ideals. To be more precise, for a given a squarefree monomial ideal $I= \left( \textbf{x}_{F_1}, \ldots, \textbf{x}_{F_r} \right)$, let
$$ P_{F_1} \cap \cdots \cap P_{F_r}= \left( \textbf{x}_{[n] \setminus T_1}, \ldots, \textbf{x}_{[n] \setminus T_s} \right),$$
and $\Delta = \langle T_1, \ldots, T_s \rangle$. Then, $I_\Delta =I$ and the equivalent condition on $I$ to be tame is to show that $T_i \cup T_j =[n]$, for all $i \neq j$.
\end{rem}


\subsection*{Polarization and tameness}
Polarization is a technique which corresponds to an arbitrary monomial ideal, a squarefree monomial ideal in a new set of variables. The construction is as follows:

Let $I$ be a monomial ideal in the polynomial ring $R=k[x_1, \ldots, x_n]$ with the (unique) minimal set of generators $\mathcal{G}(I) = \left\{ u_1, \ldots, u_r \right\}$, where $u_i = \prod_{j=1}^{n} x_j^{\alpha_{i,j}}$. Let $\alpha_i$ be the maximum exponent of the variable $x_i$ appearing in elements of $\mathcal{G}(I)$. Without loss of generality, we may assume that $\alpha_i$ are all positive. Let $S = R \left[Y_{i,j} \colon \; i = 1, \ldots , n \text{ and }  j = 2, \ldots , \alpha_i \right]$. For each monomial $u_i$, we define
\begin{equation*}
u_i^\mathscr{P}  = x_1^{\min \{ \alpha_{i,1},1 \} } Y_{1,2} \cdots Y_{1, \alpha_{i,j}}  \cdots  \ x_n^{\min \{ \alpha_{i,n}, 1 \}} Y_{n,2} \cdots Y_{n, \alpha_{i,j}}
\end{equation*}
By the choice of $\alpha_i$, $u_i^\mathscr{P} \in S$. The polarization $I^\mathscr{P}$ of $I$ is the ideal in $S$ generated by $\left\{u_1^\mathscr{P}, \ldots, u_r^\mathscr{P} \right\}$.

A monomial ideal $I$ and its polarization $I^\mathscr{P}$ share many homological and
algebraic properties. Thus, by polarization, many questions concerning monomial ideals can be reduced to squarefree monomial ideals. For example the graded Betti numbers of $I$ and $I^\mathscr{P}$ are the same (c.f \cite[Corollary 1.6.3]{HHBook}). In the following we show that if $I^\mathscr{P}$ is a tame ideal, then so is $I$, but the converse is not true as the following example shows:

\begin{ex}
Assume that $I = \left( x_1^2, x_1x_2, x_2^2 \right) \subset k[x_1,x_2]$. Then, $I^\mathscr{P} = ( x_1y_{1,2},
x_1x_2, x_2y_{2,2}) \se k[x_1, x_2, y_{1,2}, y_{2,2}]$ is not tame by Theorem~\ref{main cluts} (or
Corollary~\ref{simple graphs}). However, it follows from Theorem~\ref{main looped graph} that, the ideal $I$ is
tame.
\end{ex}

\begin{prop}
Let $I \subset R$ be a monomial ideal and $I^\mathscr{P} \subset S$ be the polarization of $I$. The following are
equivalent:
\begin{enumerate}
\item $I^\mathscr{P}$ is a tame squarefree monomial deal in $S$;
\item there exist a monomial $u \in R$ and a tame squarefree monomial ideal $I' \subset R$ such that $I= u I'$.
\end{enumerate}
In particular, if $I^\mathscr{P}$ is tame, then so is $I$.
\end{prop}

\begin{proof}
(i) \give\ (ii): Let $\cc$ be the clutter associated to $I^\mathscr{P}$. Since $I^\mathscr{P}$ is a tame
squarefree monomial ideal, it follows from Theorem~\ref{main cluts} that, $\cc$ is a uniform clutter, say
$d$-uniform, and hence $I$ is a monomial ideal generated by monomials of the same degree $d$.

Using the same notation as above, let $\alpha_i$ be the maximum exponent of the variable $x_i$ appearing in
elements of $\mathcal{G}(I)$. Again, without loss of generality, we may assume that $\alpha_i$ are all positive.
Since $I^\mathscr{P}$ is tame, we conclude from Theorem~\ref{main cluts} that, $\cc$ is complete $d$-partite and
$\{Y_{i,2}\}, \ldots, \{Y_{i,\alpha_i}\}$ are some of the partitions of $\cc$, for all $i$. So $Y_{i,2} \cdots
Y_{i,\alpha_i}$ divides all generators of $I^\mathscr{P}$, for all $i$. Equivalently, $x_i^{\alpha_i-1}$ divides
all generators of $I$, for all $i$. Let $u= x_1^{\alpha_1-1} \cdots x_n^{\alpha_n-1}$ and $I' = \left( f/u \colon
\; f \in \mathcal{G}(I) \right)$. Then by the above discussion, $I'$ is a squarefree monomial ideal in $R$ and
$I=u I'$. Let $\cc'$ be the clutter associated to $I'$. Then,
$$\cc'= \big\{ F \setminus \left\{ Y_{1,2}, \ldots, Y_{1,\alpha_1}, \ldots, Y_{n,2}, \ldots, Y_{n,\alpha_n}\right\} \colon \quad F \in \cc \big\}.$$
Since $\cc$ is tame, we may apply Theorem~\ref{main cluts}, to observe that $\cc'$ is again tame. 

(ii) \give\ (i): In this case, $I^\mathscr{P} =u^\mathscr{P}I'$. Hence, $\mathcal{G} \left(I^\mathscr{P} \right) =
\left\{ u^\mathscr{P}f \colon \; f \in \mathcal{G}(I') \right\}$. So that, for every $f\in \mathcal{G}(I)$, the
$f$-chart of $I'$ coincide with $u^\mathscr{P}f$-chart of $I^\mathscr{P}$. Since $I'$ is a tame ideal, the
$f$-chart of $I'$ is regular, for every $f \in \mathcal{G}(I')$. Hence, $f'$-chart of $I^\mathscr{P}$ is also
regular, for all $f' \in \mathcal{G} \left( I^\mathscr{P} \right)$. This means that $I^\mathscr{P}$ is a tame
ideal.

To obtain the last statement, we use the same argument as in the proof of (ii) \give (i).
\end{proof}


\section{Tame Monomial Ideals Generated in Degree at Most Two}
Squarefree monomial ideals are not the only class of monomial ideals whose vertices can be classified exactly.
Another such class is the class of monomial ideals generated in degree at most 2. In this section we investigate
tameness of these ideals.

Our first simple observation is that if $x_1^d$ is in $\mathcal{G}(I)$ for a monomial ideal $I$ and $\textbf{a}
=\left( a_1, \ldots, a_n \right)\in \supp I$ (or even in $N(I)$) with $a_i=0$ for all $1\neq i$, then $a_1\geq d$.
Now if $\textbf{p}= (d,0,0,\ldots, 0)$ is a convex combination of different points in $N(I)$, then clearly all
these points should have zero on all entries except the first one. So indeed all these points should be
$\textbf{p}$ itself, that is, $\textbf{p}$ is a vertex of $N(I)$. A similar argument, which we leave the details
to the reader, proves the following lemma.

\begin{lem}\label{vertices in deg 2}
If $x_i^\alpha\in \mathcal{G}(I)$ for some $\alpha>0$, then $x_i^\alpha$ is a vertex of $I$. Also if degree of all elements
of $\mathcal{G}(I)$ is at most two, then $u\in \mathcal{G}(I)$ is not a vertex of $I$, \ifof $u=x_ix_j$ for some $i\neq j\in [n]$ such
that $x_i^2, x_j^2\in \mathcal{G}(I)$.
\end{lem}

Suppose that $I$ is a monomial ideal and $u\in \mathcal{G}(I)$ is a vertex of $I$ such that the $u$-chart of
$\Bl{I}$ is regular. If we set $U_1=\{x_1, \dots, x_n\}$, $U_2=\{\f{u'}{u} \colon \;  u\neq u'\in
\mathcal{G}(I)\}$ and $U=U_1\cup U_2$, then by Lemma~\ref{FW crit} there is a $U'\se U$ with $|U'|=n$ and
$k[U]=k[U']$. The next lemma states the form of $U'$ in some special cases. Here by $\deg_i (f)$ for a $f\in R$,
we mean the degree of $f$ in the variable $x_i$. Also $\supp u$ means $\{i \colon \;  x_i|u\}$ for a monomial $u$
of $R$.
\begin{lem}\label{U' in some case}
Using the above notations and assumptions, if
\begin{enumerate}
\item $\deg u\leq \deg u'$ for all $u'\in \mathcal{G}(I)$ and

\item $\deg_i u\geq \deg_i u'$ for all $i\in \supp u$ and all $u'\in \mathcal{G}(I)$,
\end{enumerate}
then for each $j$ with $x_j\notin U'$, there is a $v(j)\in \supp u$ such that $U'=(U_1\cap U')\cup \{
\f{x_j}{x_{v(j)}} \colon \; x_j\notin U' \}$.
\end{lem}
\begin{proof}
Suppose that $x_j\notin U'$. Then
\begin{equation}\label{Local eq1}
x_j= x_{i_1}\cdots x_{i_r}z_1\cdots z_s
\end{equation}
where $z_i=\f{u'_i}{u}$, with all terms of the right hand side in $U'$. Note that the assumption (ii) implies that
when we write $z_i$'s in the simplest form, if  $x_l$ divides the numerator of some $z_i$, then $l\notin \supp u$.
Hence the in the product $z_1\cdots z_s$ the denominators of $z_i$'s do not cancel out. This forces $r$ to be at
least one. On the other hand, since degree of the left hand side of (\ref{Local eq1}) is one and $\deg z_i\geq 0$
(by (i)), we see that $r=1$. Again since the denominators of $z_i$'s do not cancel out in their product, we
conclude that $s\leq 1$, so $s=1$. Therefore, $z_1=\f{x_j}{x_{i_1}}\in U'$ and $i_1\in \supp u$. Thus if we set
$v(j)=i_1$, then the result is established.
\end{proof}

\begin{prop}\label{degree 1}
Suppose that $\mathcal{G}(I)= \{x_1, x_2, \ldots, x_t, u_1, \ldots, u_s\}$, where $\deg u_i\geq 2$. If $t,s\geq
2$, then the $x_i$-chart of $\Bl{I}$ is not regular for $i\in [t]$.
\end{prop}
\begin{proof}
We assume $i=1$ and that the $x_1$-chart is regular. According to Lemma~\ref{vertices in deg 2}, $x_1$ is a vertex
and clearly satisfies the conditions of Lemma~\ref{U' in some case}. Thus using the notations of Lemma~\ref{U' in
some case}, it easily follows that $U'=\{x_1,\f{x_2}{x_1}, \ldots, \f{x_t}{x_1}, x_{t+1}, \ldots, x_{t+s}\}$. Now
$\f{u_1}{x_1}$ should be a product of elements of $U'$. Clearly this product contains a term of the form
$\f{x_j}{x_1}$ for $1< j\leq t$ which means $x_j|u_1$, a contradiction with the minimality of $\mathcal{G}(I)$.
\end{proof}

\begin{cor}\label{deg<=2 uniform}
Suppose that $I$ is a tame monomial ideal generated in degree at most 2. Then either all elements of $\mathcal{G}(I)$ have
degree 1 or all of them have degree 2.
\end{cor}

Corollary~\ref{deg<=2 uniform} together with Theorem~\ref{main cluts} imply that if $I$ is a monomial ideal and either is squarefree or generated in
degree at most 2, then it is generated in one degree. This poses the question ``Does there exist a tame monomial
ideal not generated in one degree?'' The answer is yes, as the following example shows.
\begin{ex}
Let $I=\lg x^2, y^3,xy \rg$. Then it is easy to see that the three generators of $I$ are vertices. Also the charts
of $\Bl{I}$ corresponding to $x^2$, $y^3$ and $xy$ are $k[x, \f{y}{x}]$, $k[y, \f{x}{y^2}]$ and $k[\f{y^2}{x},
\f{x}{y}]$, respectively. Hence according to Lemma~\ref{FW crit}, $I$ is tame.
\end{ex}

Note that any ideal generated with monomials of degree one is clearly tame. Thus we focus on ideals $I$ where
$\mathcal{G}(I)$ consists of monomials of degree 2. Such an ideal could be represented by the edge ideal of an undirected
graph with loops (but no multiple edges). For such a graphs $G$, the edge ideal $I(G)$ is generated by $\textbf{x}_{e}$'s
where $e$ is an edge of $G$. Here if $e$ is a loop on $i$, we set $\textbf{x}_e=x_i^2$. Note that if $G$ is simple, then it
is a clutter and by setting $d=2$ in Theorem~\ref{main cluts}, we get:

\begin{cor}\label{simple graphs}
Let $G$ be a simple graph. Then $G$ is tame \ifof $G$ is the disjoint union of a complete bipartite graph and some
isolated vertices.
\end{cor}

Thus, in the rest of this section  we use the following notation.
\begin{nota}\label{what is G}
In the sequel, $G$ is assumed to be a graph on $[n]$ where the vertices $1, \ldots, r$ have loops and $r+1,
\ldots, n$ do not have loops with $r\geq 1$.
\end{nota}

Note that since $1$ has a loop we have $1\in \N_G(1)$.
\begin{lem}\label{loop chart}
Using Notation \ref{what is G} and for $i\in [r]$, the $x_i^2$-chart of $\Bl{I(G)}$ is regular \ifof for all edges
$e$ of $G$ (including loops), we have $e\se \N_G(i)$.
\end{lem}
\begin{proof}
(\give): We assume $i=1$. By Lemma~\ref{vertices in deg 2}, $x_1^2$ is a vertex of $I(G)$ and clearly satisfies the
conditions of Lemma~\ref{U' in some case}. Let $U$ be as in the statement right before Lemma~\ref{U' in some case}. It follows
that $k[U]$ is generated over $k$ by a set $U'$ of some variables and some monomials of the form
$u_j=\f{x_j}{x_1}\in U$. So $u_j=\f{\textbf{x}_{e'}}{x_1^2}$ for some edge $e'$ of $G$ which should be $e'= \{1,j\}$.
Therefore $j\in \N(1)$. Now for each edge $e$ of $G$, $\f{x_e}{x_1^2}$ should be a product of elements of $U'$ and
since $\deg \f{\textbf{x}_e}{x_1^2}=0$ we should have $\f{\textbf{x}_e}{x_1^2}= \f{x_i}{x_1}\f{x_j}{x_1}$ for some $i,j\in \N(1)$,
and $e\se \N(1)$.

(\rgive): It is easy to see that the coordinate ring is generated by $\f{x_j}{x_i}$ for $j\in \N(i)$ together with
$x_l$ with $l\notin \N(i)$ or $l=i$.
\end{proof}

In the following by a \emph{looped star graph} we mean an star graph with an additional loop on the central vertex
and by a \emph{looped complete graph}, we mean a complete graph with loops on each vertex.
\begin{thm}\label{main looped graph}
Suppose that $G$ is a graph with at least one loop. Then $G$ is tame \ifof $G$ is a disjoint union of some
isolated vertices with either a looped star graph or a looped complete graph.
\end{thm}
\begin{proof}
We use Notation \ref{what is G}. Clearly we can assume that $G$ has no isolated vertices. (\give): By
Lemma~\ref{loop chart}, we see that the induced subgraph of $G$ on $[r]$ (denoted $G_1$) is a looped complete
graph. Also since for each $r<j$ the vertex $j$ is in some edge and again by Lemma~\ref{loop chart}, $j\in \N(i)$
for each $i\leq r$.

Assume that there are vertices $r< i\neq j\leq n$ such that $e=\{i,j\}$ is an edge of $G$. Then by
Lemmata~\ref{vertices in deg 2} and \ref{U' in some case}, the coordinate ring $S$ of the $e$-chart of $\Bl{I(G)}$
is generated as a $k$ algebra by a set $U'$ of some variables and some monomials of the form $\f{x_l}{x_{v(l)}}$
for some $l\in [n]$ and $v(l)\in e$. Since the only elements of $U'$ with degree zero are in the latter form and
$u=\f{x_1^2}{x_ix_j}\in S$, we conclude that $u= \f{x_l}{x_i} \f{x_{l'}}{x_j}$ for some $\f{x_l}{x_i},
\f{x_{l'}}{x_j}\in U'$. But then $l=l'=1$ and hence $v(l)=i=j$ a contradiction. Thus the induced subgraph of $G$
on vertices $r+1, \ldots, n$ has no edges.

Consequently, $G$ is the union of a looped complete graph $G_1$ and all edges $\{i,j\}$ with $i\leq r<j\leq n$. If
$r=n$, we are finished. Hence assume that $r<n$ and consider the chart corresponding to $e=\{1,n\}$. One can
readily check that its coordinate ring is
 \notag
\begin{align}
 S'= & k\l[ \{x_1,\ldots, x_n\} \cup \{ \f{x_ix_j}{x_1x_n} \colon \; i\in [r], j\in [n] \}  \r]\\
  =& k\l[\{ \ x_n\} \cup \{ \f{x_j}{x_n}\colon \;  1\leq j\leq n-1 \} \cup \{ \f{x_i}{x_1} \colon \; 1<i \leq r \} \r].
\end{align}
Note that the second set of generators $U'$ of $S'$ is minimal and as $x_e$ is a vertex of $I(G)$ by Lemma~\ref{vertices
in deg 2}, we should have $1+(n-1)+(r-1)=|U'|=n$ or $r=1$. This means that $G$ is a looped star graph as required.

(\rgive): If $G$ is looped complete graph, then the according to Lemma~\ref{vertices in deg 2} the only vertices
of $I(G)$ are $x_i^2$ for $i\in [r]$ and the corresponding charts are regular by Lemma~\ref{loop chart}. If $G$ is
a looped star graph, then the $x_1^2$-chart is regular by Lemma~\ref{loop chart}. For the other charts, say
corresponding to $x_1x_n$, the minimal set of generators found above for $S'$ has cardinality $n$ and hence the
result follows by Lemma~\ref{FW crit}.
\end{proof}

The following summarizes our results in this section.  Note that in this result, (i) corresponds to ideals
generated in degree 1, (ii) to looped complete graphs, (iii) to looped star graphs and (iv) to simple
graphs.
\begin{cor}\label{main deg<=2}
Assume that $I$ is an ideal of $R$ generated by monomials of degree at most two. Then the blowup of $\bbA_k^n$
along $I$ is regular \ifof one of the following holds.
\begin{enumerate}
\item There exists a nonempty $F\se [n]$ such that $I=P_F$.

\item There exists a nonempty $F\se [n]$ such that $I=P_F^2$.

\item There exist  nonempty $F\se [n]$ and $i\in F$ such that $I=x_iP_F$.

\item There exist nonempty $F_1,F_2\se [n]$ with $F_1\cap F_2=\tohi$ and $I=P_{F_1}P_{F_2}$.
\end{enumerate}
\end{cor}

\section{Rees algebra of tame clutters}

In this section we  find a generating set for the presentation ideal or toric ideal of the coordinate ring of an
affine chart of the blowup $\mathbb{A}_k^n$ along the circuit ideal of a complete $d$-partite $d$-uniform
clutter. By using this presentation, we look closely at the defining ideal of the Rees algebra of a tame complete
$d$-partite $d$-uniform clutter.

If  $I$ is the edge ideal of a graph, R. H. Villarreal shows that $I$ is of fiber type \cite[Theorem 3.1]{RV1}. He
gives an explicit description of the defining ideal  of the Rees algebra of any squarefree monomial ideal
generated in degree 2.  Villarreal give an example of a clutter to show that his arguments do not extend for
monomial ideals generated in higher degree  than 2 \cite[Example 3.1]{RV1}. If $I$ is the edge ideal of a complete
bipartite graph he finds the defining equations of the Rees algebra of I in terms of 2-minors of a ladder \cite[Proposition 2.1]{RV2}.
Here we generalize the latter result of Villarreal. More precisely, we find the defining ideal of the Rees algebra
of a tame squarefree monomial ideal . This result shows that such ideals are of fiber type.

\begin{nota}
We let $\cc$ to be a complete $d$-partite $d$-uniform clutter with the $d$-partition $\{V_i \colon \; i\in [d]\}$ and $e\in
\cc$. Also $S_{e}$ stands  for the coordinate ring of the $e$-chart of $\mathrm{Bl}_{I(\cc)} R$.  Consider the
ring homomorphism
$$\phi_{e}\colon S= R[\{T_{e'}\colon \; e\neq e'\in \cc\}]\surjects S_{e}$$ that sends  $T_e$ to $\f{\mathbf {x}_{e'}}{\mathbf{x}_{e}}$. Set
$J_{e}=\ker \phi_{e}$. Moreover, for $e\neq e'\in \cc$ we fix a vertex $v(e,e')\in e\sm e'$ such that $v(e,e')$
and $v(e',e)$ lie in the same partition and by $v_e(j)$ we mean the only vertex of $e$ in the same partition as
the vertex $j$. Finally we denote the circuit of $\cc$ obtained from $e$ by replacing $j$ instead of $v_e(j)$  by
$e(j)$.
\end{nota}

Assume that $j= v(e,e')\in V_i$. Then since $j'=v(e',e)$ is in $V_i\cap e'$ we have $j'=v_{e'}(j)$ and similarly
$j=v_e(j')$. In this case, $e(j')$ and $e'(j)$ are the circuits obtained from $e$ and $e'$ respectively, by
``swapping'' those vertices of $e$ and $e'$ which lie in the $i$'th partition. For example, $e(j')= (e\cup\{j'\})
\sm \{v_e(j')\}= (e\cup\{j'\}) \sm \{j\}$.
\begin{prop}\label{ker chart}
Using the above notations and assumptions, for each $e\in \cc$, $J_e$ is generated by $G=G_1\cup G_2$ where $G_1$
is the set of all binomials of the form $x_i-x_rT_{e(i)}$ with $i\in [n]\sm e$ and $r=v_{e}(i)$ and $G_2$ is the
set of all binomials of the form $T_{e'}-T_{e(j')}T_{e'(j)}$ where $j= v(e,e')$ and $j'=v(e',e)$, for $e'\in \cc$
with $|e'\sm e|>1$.
\end{prop}
\begin{proof}
For simplicity we assume that $e=[d]$. It is routine to check that the stated binomials map to zero under
$\phi_e$, that is, $G\se \ker \phi_e$. We use $T$ to denote the set of indeterminates $\{T_{e'}: e\neq e'\in
\cc\}$ and set $T|_l= \{T_{e'}: e\neq e'\in \cc, |e'\sm e|\leq l\}$. Also by $G_{2l}$ we mean the subset of $G_2$
consisting of those binomials with $|e'\sm e|\leq l$. In particular, $G_{21}=\tohi$. Note that for a binomial
$T_{e'}-T_{e(j')}T_{e'(j)}$ in $G_2$, $|e(j')\sm e|=1$ and $|e'(j)\sm e|=|e'\sm e|-1$.

Let $J$ be the ideal of $S$ generated by the stated binomials, $J'$ the ideal of $S'=k[x_1,\ldots, x_d][T]$
generated by $G_2$ and $J'_l$ the ideal of $S'_l=k[x_1,\ldots, x_d][T|_l]$ generated by $G_{2l}$. Then it is easy
to see that the homomorphism $S\to \f{S'}{J'}$ which maps $x_i$ to $x_rT_{e(i)}$ with $i$ and $r$ as in the
statement of the theorem, induces an isomorphism $\f{S}{J}\to \f{S'}{J'}$. Similarly:
 $$\f{S}{J}\cong \f{S'}{J'}=\f{S'_{d-1}}{J'_{d-1}} \cong \f{S'_{d-2}}{J'_{d-2}} \cong \cdots \cong
 \f{S'_1}{J'_1}=S'_1.$$
Because $S'_1$ is a polynomial ring with $n$ indeterminates, $J$ is a prime ideal with $\dim \f{S}{J}=n$. But
$\ker\phi_e$ is also a prime ideal with $\f{S}{\ker \phi_e}$ having dimension $n$ and $J\se \ker \phi_e$.
Consequently, $J=\ker \phi_e$.
\end{proof}
\begin{thm}\label{Rees Algebra Clutter}
Let $\cc$  be a complete $d$-partite $d$-uniform clutter with $d$-partition $\{V_i:i\in [d]\}$. Then
\[
\mathcal{R}_R(I(\cc)) \simeq \frac{R \left[ T_e\colon \;\ e\in \cc \right]}{\mathcal{A}}
\]
where $\mathcal{A}$ is generated by the set of all binomials of the form $T_ex_i- x_rT_{e(i)}$ with $e\in \cc$,
$i\in [n]\sm e$ and $r=v_{e}(i)$ together with those of the form $T_eT_{e'}-T_{e(j')}T_{e'(j)}$ where $j= v(e,e')$
and $j'=v(e',e)$, for $e\neq e'\in \cc$ with $|e'\sm e|>1$. In particular, the circuit ideal of a  complete
$d$-partite $d$-uniform clutter is of fiber type.
\end{thm}
\begin{proof}
First note that the set of generators of $\mathcal{A}$ is obtained by homogenization of the generators of  affine
chart $J_e$ in the variable $T_e$ and taking union over all $e\in \cc$.  By Proposition \ref{ker chart} and the
proof of Proposition \ref{charts}(ii), clearly $\mathcal{A}$ belongs to the defining ideal $\mathcal{J}$ of
$\mathcal{R}_R(I(\cc))$. Conversely, let $F\in \mathcal{J}$. Then the dehomogenizing of $F$ with respect to a
variable $T_e$, which has appeared in $F$, lies in at least one affine chart. Thus again homogenizing it with
respect to $T_e$, we get the assertion.
\end{proof}
\begin{ex}
Let $R=k[x_1,x_2,y_1,y_2,z]$. By Theorem  \ref{regular chart} the ideal $I=(x_1,x_2)(y_1,y_2)(z)$ is tame. By Theorem \ref{Rees Algebra Clutter}, the defining ideal of the Rees algebra of $I$ is generated by 
\[x_2T_1-x_1T_3,\ x_2T_2-x_1T_4,\ y_1T_4-y_2T_3,\ y_2T_1-y_1T_2,\  T_2T_3-T_1T_4. \]
\end{ex}
Note that the defining ideal of the Rees algebra of ideals in cases (i) and (iv) in Corollary \ref{main deg<=2}, can be determined by Theorem \ref{Rees Algebra Clutter}. Also in case (iii), the defining equation of the $x_iP_F$ is the same as those of $P_F$. Finally, in the case (ii), the defining ideal of the Rees algebra of $P_F^2$ can be found in \cite[Theorem 1]{Barshay}. Consequently, tame monomial ideals generated in degree at most $2$ are of fiber type.  

\end{document}